\documentclass[11pt]{article}
\usepackage{amsfonts}
\usepackage{amssymb}
\usepackage{amstext}
\usepackage{amsmath}
\usepackage{amsthm,verbatim}
\usepackage{xspace}
\usepackage{graphicx}
\usepackage{algorithm}

\usepackage{typearea}
\typearea{14}

\usepackage{color}
\definecolor{Darkblue}{rgb}{0,0,0.4}
\definecolor{Brown}{cmyk}{0,0.81,1.,0.60}
\definecolor{Purple}{cmyk}{0.45,0.86,0,0}
\usepackage[breaklinks]{hyperref}
\hypersetup{colorlinks=true,
            citebordercolor={.6 .6 .6},linkbordercolor={.6 .6 .6},
            citecolor=black,urlcolor=blue,linkcolor=black,pagecolor=black}
\newcommand{\lref}[2][]{\hyperref[#2]{#1~\ref{#2}}}
\usepackage{hyperref}

\let\oldfootnoterule=\footnoterule
\def\footnoterule{\vspace{-1ex}\oldfootnoterule\vspace{1ex}}

\newtheorem{lemma}{Lemma}
\newtheorem{theorem}{Theorem}
\newtheorem{definition}{Definition}

\makeatletter
\def\newxxxproof#1{\@nxxxprf{#1}}

\def\@nxxxprf#1#2{\@xnxxxprf{#1}{#2}}

\def\@xnxxxprf#1#2{\expandafter\@ifdefinable\csname #1\endcsname
\global\@namedef{#1}{\@xxxprf{#1}{#2}}\global\@namedef{end#1}{\@endxxxproof}}

\def\@xxxprf#1#2{\@xxxxprf{#1}{#2}}

\def\@xxxxprf#1#2{\@beginxxxproof{#2}{\csname the#1\endcsname}\ignorespaces}

\def\@beginxxxproof#1{\rm \trivlist \item[\hskip \labelsep{\sc #1.\/}]}
\def\@endxxxproof{\outerparskip 0pt\endtrivlist}

\newxxxproof{@xxxproof}{Main Theorem}

\makeatother

\newcommand{\f}[2][xxx]{\ifthenelse{\equal{#1}{xxx}}{{f}_{#2}}{{f}_{#2}[#1]}}
\renewcommand{\c}[2][xxx]{\ifthenelse{\equal{#1}{xxx}}{{c}_{#2}}{{c}_{#2}[#1]}}
\renewcommand{\o}[2][xxx]{\ifthenelse{\equal{#1}{xxx}}{{o}_{#2}}{{o}_{#2}[#1]}}

\newcommand{\ival}{{\mathcal I}}

\newcounter{note}[section]

\def\beq{\begin{eqnarray}}
\def\eeq{\end{eqnarray}}
\def\ben{\begin{enumerate}}
\def\een{\end{enumerate}}
\def\beqs{\begin{eqnarray*}}
\def\eeqs{\end{eqnarray*}}
\newcommand{\N}{\mathbb{N}}

\newcommand{\R}{\mathbb{R}}

\newcommand{\eps}{\epsilon}

\newcommand{\de}{\delta}

\newcommand{\E}{\mathbb{E}}
\newcommand{\p}{\mathbb{P}}

\newcommand{\lab}{\label}

\newcommand{\ra}{\rightarrow}

\newcommand{\HH}{\mathbb H}

\newcommand{\ii}{\iota}

\newcommand{\hs}{\mathrm{hsiz}}
\newcommand{\hcap}{\mathrm{hcap}}

\newcommand{\Half}{\HH}
\newcommand{\rect}{{\mathcal R}}
\newcommand{\ball}{{\mathcal B}}
\newcommand{\area}{\mathrm{area}}
\renewcommand {\Im} {{\rm Im}}

\begin{document}

\title{\Large  A Geometric Interpretation of Half-Plane Capacity}
\author{Steven Lalley\\
Department of Statistics\\ University of Chicago
\thanks{Research supported by National Science Foundation grant
DMS-0805755.}
\and
Gregory Lawler\\
Department of Mathematics\\ University of Chicago
\thanks{Research supported by National Science Foundation grant
DMS-0734151.}
 \and Hariharan Narayanan\\ Department of Computer Science \\
University of Chicago}
\date{}
\maketitle

\begin{abstract} 
Let $A$ be a bounded, relatively closed subset of the upper
half plane $\Half$ whose complement in $\Half$ is simply connected. If $B_t$
is a standard complex Brownian motion and $\tau_A = \inf\{t
\geq 0: B_t \not \in \Half \setminus A\}$, the half-plane capacity $\hcap(A)$ is defined as 
\[  \hcap(A) := \lim_{y \rightarrow \infty}y\, \E^{iy}
  \left[\Im(B_{\tau_A})\right]. \] This quantity arises in the study of Schramm-Loewner Evolutions (SLE). In this note, we show that $\hcap(A)$ is comparable to a more geometric quantity $\hs(A)$ that we define to be the $2$-dimensional Lebesgue measure of the
union of all balls tangent to $\R$ whose centers belong to $A$. Our main result is that
\[ \frac{1}{66}\, {\hs(A)}  < {\hcap(A)} \leq  \frac{7}{2\pi}\, {\hs(A)} .\] 
\end{abstract}
\section{Introduction}
Suppose $A$ is a bounded, relatively closed subset of the upper
half plane $\Half$.  We call $A$ a compact $\Half$-hull if $A$
is bounded and $\Half \setminus A$ is simply connected.  The
{\em half-plane capacity} of $A$, $\hcap(A)$, is defined in
a number of equivalent ways (see
\cite{lawler}, especially Chapter 3).  If $g_A$ denotes the
unique conformal transformation of $\Half \setminus A$
onto $\Half$ with $g_A(z) = z + o(1)$ as $z \rightarrow
\infty$, then $g_A$ has the expansion
\[  g_A(z) = z + \frac{\hcap(A)}{z} + O(|z|^{-2}), \;\;\;\;
  z \rightarrow \infty.\]
Equivalently, if $B_t$
is a standard complex Brownian motion and $\tau_A = \inf\{t
\geq 0: B_t \not \in \Half \setminus A\}$,
\[  \hcap(A) = \lim_{y \rightarrow \infty}y\, \E^{iy}
  \left[\Im(B_{\tau_A})\right]. \]
Let $\Im[A] = \sup\{\Im(z): z \in A\}$.  Then if $y \geq
\Im[A]$, we can also write
\[       \hcap(A) = \frac 1 \pi \int_{-\infty}^\infty \E^{x+iy}
   \left[\Im(B_{\tau_A})\right]\, dx. \]
These last two definitions do not require $\Half \setminus A$
to be simply connected, and the latter definition does not require
$A$ to be bounded but only that $\Im[A]<\infty$.

 For $\Half$-hulls
(that is, for $A$ for which $\Half \setminus A$ is simply
connected), the half-plane capacity is comparable to a more
geometric quantity that we define.  This is not new (the second author learned it from
 Oded Schramm in oral communication), but
we do not know of a proof in the literature.  In this
note, we prove the fact giving (nonoptimal) bounds on the
constant.
 We start with the definition of the geometric
quantity.

\begin{definition}
 For an $\Half$-hull
$A$, let $\hs(A)$ be the $2$-dimensional Lebesgue measure of the
union of all balls centered at points in $A$ that are tangent to the real line.
In other words
\[   \hs(A) = \area \left[\bigcup_{x+iy \in A} \ball(x+iy,y)
  \right], \]
where $\ball(z,\epsilon)$ denotes the disk of radius $\epsilon$ about $z$.
\end{definition}

In this paper, we prove the following.

\begin{theorem} \label{main} For every $\Half$-hull $A$,
\[ \frac{1}{66}\, {\hs(A)}  < {\hcap(A)} \leq  \frac{7}{2\pi}\, {\hs(A)} .\]
\end{theorem}
\section{Proof of Theorem~\ref{main}}
It suffices to prove this for
weakly bounded $\Half$-hulls, by which we mean $\Half$-hulls
$A$ with $\Im(A) < \infty$ and such that for each $\epsilon > 0$,
the set $\{x+iy: y > \epsilon\}$ is bounded.  Indeed, for $\Half$-hulls
that are not weakly bounded, it is easy to verify that
$\hs(A) = \hcap(A) = \infty$.

We start with a simple inequality that is implied but not explicitly stated
in \cite{lawler}.
Equality is achieved when $A$ is a vertical line segment.
\begin{lemma}  If $A$ is an $\Half$-hull, then
\begin{equation}  \label{jul13.2}
     \hcap(A) \geq \frac {\Im[A]^2}{2} .
\end{equation}
\end{lemma}

\begin{proof}  Due to the continuity of $\hcap$ with respect to the Hausdorff metric on $\Half$-hulls, it suffices to prove the result for $\Half$-hulls that are path-connected. Further, by the monotonicity of $\hcap$ under containment, $A$ can be assumed to be of the
form $\eta(0,T]$ where $\eta$ is a
simple curve with $\eta(0+)  \in \R$, parameterized so that
$\hcap[\eta(0,t]) = 2t$. In particular, $T = \hcap(A)/2$.
If $g_t = g_{\eta(0,t]}$, then $g_t$
satisfies the Loewner equation
\begin{equation}  \label{jul13.1}
         \partial_t g_t(z) = \frac{2}{g_t(z) - U_t} , \;\;\;\;
   g_0(z) = z,
\end{equation}
where  $U:[0,T] \rightarrow \R$ is continuous.  Suppose $
\Im(z)^2 > 2 \, \hcap (A)$ and let $Y_t = \Im[g_t(z)]$.  Then
\eqref{jul13.1} gives
\[   - \partial_t Y_t^2 \leq \frac{4 Y_t}{|g_t(z) - U_t|^2} \leq 4, \]
which implies
\[    Y_T^2 \geq Y_0^2 - 4T > 0. \]
This implies that $z \not\in A$, and hence $\Im[A]^2 \leq
 2 \, \hcap(A).$
\end{proof}

The next lemma is a variant of the vital covering lemma.  If $c > 0$
and
 $z = x+iy \in \Half$,  let
\[  \ival(z,c) = (x - cy, x+cy) , \]
\[              \rect(z,c) = \ival(z,c) \times (0,y]
   = \{x' + iy': |x'-x| < c y , 0 < y' \leq y\}. \]

\begin{lemma}  \label{todaylem.1}
Suppose $A$ is a weakly
bounded $\Half$-hull and $c > 0 $.  Then
there exists a finite or countably infinite
sequence of points $\{z_1 = x_i + iy_1,z_2 = x_2 + iy_2,,\ldots\} \subset
A$ such that:
\begin{itemize}
\item  $y_1 \geq y_2 \geq y_3 \geq \cdots $;
\item the intervals $\ival(x_1,c),\ival(x_2,c),\ldots$
are disjoint;
\item
\begin{equation}  \label{today.1}
           A \subset \bigcup_{j=1}^\infty  \rect(z_j,
    2 c).
\end{equation}
\end{itemize}
\end{lemma}

\begin{proof}  We define the points recursively.  Let $A_0 = A$ and
given $\{z_1,\ldots,z_j\}$, let
\[    A_j = A \setminus \left[\bigcup_{k=1}^{j} \rect(z_j,2c)\right]
.\]
If $A_j = \emptyset$ we stop, and if
$A_j \neq \emptyset$,we  choose $z_{j+1} = x_{j+1}
 + iy_{j+1}  \in A$ with  $y_{j+1} = \Im[A_j]$. Note that if $k \leq j$,
then $|x_{j+1} - x_k| \geq 2 \, c \, y_k \geq c \,(y_k+
 y_{j+1})$
and hence $\ival(z_{j+1},c) \cap \ival(z_k,c) = \emptyset.$
Using the weak boundedness of $A$, we can see that $y_{j} \rightarrow
0$ and hence \eqref{today.1} holds.
\end{proof}

\begin{lemma} \label{todaylem.2}
 For every $c > 0$, let $$\rho_c := \frac{ 2 \sqrt{2}}{\pi}
\,  \arctan\left( e^{-\theta} \right), \;\;\;\;
  \theta = \theta_c = \frac{ \pi}{4c} .$$
Then, for any $c > 0$, if $A$ is a
weakly bounded $\Half$-hull and $x_0 +iy_0 \in A$ with
$y_0 = \Im(A)$,  then
\[  \hcap(A) \geq  \rho_c^2 \, y_0^2 +
    \hcap \left[A \setminus \rect(z,2c) \right]. \]
\end{lemma}

\begin{proof}
By scaling and invariance under real translation,
we may assume that $\Im[A] = y_0 = 1$ and $x_0 = 0$.
Let $S = S_c $ be defined to be the set of all points $z$
of the form $x + iuy$ where $x+iy \in A \setminus \rect(i, 2c)$ and $ 0 < u \leq 1$.

Clearly, $S \cap A = A \setminus \rect(i, 2c)$.

Using the capacity inequality
\cite[(3.10)]{lawler}
\begin{equation}  \label{capin}
 \hcap(A_1 \cup A_2) -  \hcap(A_2) \leq \hcap(A_1)
   - \hcap (A_1 \cap A_2),
\end{equation}
   we see that
   \[\hcap(S \cup A) - \hcap(S) \leq \hcap(A) - \hcap(S \cap A).\]
Hence, it suffices to show that
\[   \hcap(S \cup A) - \hcap( S) \geq \rho_c^2 . \]
Let  $f$  be the conformal map of $\HH \setminus S$ onto $\HH$
such that $ z - f(z) = o(1)$ as $z \rightarrow \infty$.
Let $S^* := S \cup A$.
By  properties of halfplane capacity \cite[(3.8)]{lawler} and
\eqref{jul13.2},
\[ \hcap(S^*) - \hcap( S) = \hcap[f(S^* \setminus S)] \geq \frac{\Im[f(i)]^2}{2}. \]
Hence, it suffices to prove that
\begin{equation}  \label{jul13.3}
 \Im[f(i)] \geq \sqrt 2 \, \rho   = \frac 4 \pi \,
 \arctan\left( e^{-\theta}\right).
\end{equation}

 By construction, $S \cap \rect(z, 2c) = \emptyset$.
Let $V = (-2c,2c) \times (0,\infty) = \{x+iy: |x| < 2c, y > 0\}$ and
let $\tau_V$ be the first time that a Brownian motion leaves the
domain.
Then \cite[(3.5)]{lawler}, \[
 \Im[f(i)]   = 1-   \E^{i}\left[\Im(B_{\tau_{S}})\right]
\geq  \p\left\{ B_{\tau_{S}} \in [-2c,2c]  \right\} \\
  \geq   \p\left\{ B_{\tau_{V}} \in [-2c,2c]  \right\}
.\]
The map $\Phi(z) = \sin \left(\theta z\right)$ maps $V$ onto
$\Half$ sending $[-2c,2c]$ to $[-1,1]$ and
$  \Phi(i) = i \, \sinh \theta.$
Using conformal invariance of Brownian motion and the Poisson kernel
in $\Half$, we see that
\[  \p\left\{ B_{\tau_{V}} \in [-2c,2c]  \right\} =
   \frac{2}{\pi} \arctan \left( \frac{1}{\sinh \theta} \right) = \frac
 4 \pi \, \arctan \left(e^{-\theta}\right). \]
The second equality uses the double angle formula for the tangent.
\end{proof}

\begin{lemma}\label{l:2last}  Suppose $c > 0$ and $x_1+iy_1,
x_2 + iy_2,\ldots$
are as in Lemma \ref{todaylem.1}.  Then
\begin{equation}  \label{bast.1}
     \hs(A) \leq [\pi + 8c]  \sum_{j=1}^\infty y_j^2.
\end{equation}
If $c \geq 1$, then
\begin{equation} \pi  \sum_{j=1}^\infty y_j^2
      \leq \hs(A). \end{equation}
\end{lemma}

\begin{proof}  A simple geometry exercise shows that
\[ \area \left[ \bigcup_{x+iy \in \rect(z_j,2c)}
     \ball(x+iy,y) \right] =  [\pi + 8c] \, y_j^2. \]
Since
\[   A \subset \bigcup_{j=1}^\infty \rect(z_j,2c), \]
the upper bound in \eqref{bast.1} follows.  Since
$c \geq 1$, and the intervals $\ival(z_j,c)$ are disjoint,  so are the
disks $\ball(z_j,y_j)$. Hence,
\[      \area\left[\bigcup_{x+iy \in A} \ball(x+iy,y)\right]
\geq \area\left[
                 \bigcup_{j=1}^\infty  \ball(z_j,y_j)\right]
  = \pi \sum_{j=1}^\infty y_j^2 .\]
\end{proof}

\begin{proof}[Proof of Theorem \ref{main}]
 Let $V_j = A \cap \rect(z_j,c)$.
Lemma \ref{todaylem.2} tells us that
\[  \hcap\left[\bigcup_{k=j}^\infty V_j\right]
  \geq \rho_c^2 \, y_j^2 +
    \hcap\left[\bigcup_{k=j+1}^\infty  V_j\right], \]
and hence
\[  \hcap(A) \geq  \rho_c^2 \, \sum_{j=1}^\infty y_j^2. \]
Combining this with the upper bound in \eqref{bast.1} with any $c > 0$ gives
\[   \frac{\hcap(A) } {\hs(A)}
\geq \frac{\rho^2_c}{\pi + 8c}.\]
Choosing $c = \frac{8}{5}$ gives us
\[   \frac{\hcap(A) } {\hs(A)}
> \frac{1}{66}.\]

For the upper bound, choose a covering as in Lemma \ref{todaylem.1}
with $c = 1$.
Subadditivity and scaling give
\[ \hcap(A) \leq \sum_{j=1}^\infty \hcap\left[\rect(z_j, 2y_j)
\right]
   = \hcap[\rect(i,2)] \, \sum_{j=1}^\infty
   y_j^2. \] Combining this with the lower bound in
\eqref{bast.1} gives
\[   \frac{\hcap(A)}{\hs(A)} \leq \frac{\hcap[\rect(i,2)]}
   \pi. \]
Note that  $\rect(i,2)$ is the union of two real
translates of $\rect(i,1)$, $\hcap[\rect(i,2)] \leq
  2 \, \hcap[\rect(i,1)]$ whose intersection is the
interval $(0,i]$.  Using \eqref{capin},  we see that
\[  \hcap(\rect(i,2)) \leq 2 \, \hcap( \rect(i,1))
         - \hcap ((0,i]) = 2 \, \hcap( \rect(i,1))
 - \frac 12. \]
But $\rect(i,1)$ is
strictly contained
in $ A' := \{z \in \Half: |z| \leq \sqrt 2\}$, and hence
\[      \hcap[\rect(i,1)] < \hcap(A') = 2
 . \]
The last equality can be seen by considering $h(z) =   {z} + 2z^{-1} $
which maps $\Half \setminus A'$ onto $\Half$.  Therefore,
\[   \hcap[\rect(i,2)] < \frac 72, \] and hence
\[   \frac{\hcap(A)}{\hs(A)} \leq \frac{7}
   {2\pi}. \]
\end{proof}

%
\end{document}